\documentclass[12pt]{amsart}
\usepackage{amsmath}
\usepackage{amssymb}
\newtheorem{theorem}{Theorem}
\newtheorem{proposition}[theorem]{Proposition}

\newcommand{\C}{\mathbb C}
\newcommand{\D}{\mathbb D}

\def\Re{\operatorname{Re}}

\title{Convex domains with locally Levi-flat boundaries}

\author{Nikolai Nikolov, Pascal J. Thomas}

\address{Institute of Mathematics and Informatics\\ Bulgarian Academy
of Sciences\\ Acad. G. Bonchev 8, 1113 Sofia,
Bulgaria}\email{nik@math.bas.bg}

\address{Universit\'e de Toulouse\\ UPS, INSA, UT1, UTM \\
Institut de Math\'e\-matiques de Toulouse\\
F-31062 Toulouse, France} \email{pthomas@math.univ-toulouse.fr}



\begin{document}

\begin{thanks}{This note was written during the stay of the
first-named author as a guest professor at the Paul Sabatier University, Toulouse in
November--December 2011.}
\end{thanks}

\begin{abstract} It is shown that a domain in $\C^n$ which is locally convex and has $\mathcal C^1$-smooth
Levi-flat boundary is locally linearly equivalent to a Cartesian product of a planar domain and $\C^{n-1}.$
\end{abstract}

\maketitle

A hypersurface $M\subset\C^n$ is called Levi-flat near a $\mathcal C^1$-smooth point $p\in M$ if $M$ divides
$\C^n$ near $p$ into two pseudoconvex domains (say $M^+$ and $M^-$). By \cite{A}, $M$ is Levi-flat near $p$
if and only $M$ admits a (unique) $\mathcal C^1$-smooth foliation near $p$ by one-codimensional complex manifolds.
It is well-know that if $M$ is real-analytic, then it is locally biholomorphic to a real hyperplane.

We have more in the $\mathcal C^1$-smooth ``convex" case.

\begin{proposition}\label{1} Let $M\subset\C^n$ be a Levi-flat hypersurface near a
$\mathcal C^1$-smooth point $p\in M$ such that $M^+$ is convex. Then $M$ near $p$ is linearly equivalent to
the Cartesian product of a planar curve and $\C^{n-1}.$
\end{proposition}

\begin{proof} After an affine change of the variables, we may assume that $p=0$
and the (real) tangent hyperplane to $M$ at $0$ is $\{\Re z_1=0\}.$ In what follows, $\delta$ will
stand for a positive constant which will be shrunk as needed in the course of the proof.
By convexity $T_\delta=\{z_1=0\}\cap\D^n(0,\delta)\subset\overline{M^-}$ ($\D^n$ denotes a polydisc),
and by $\mathcal C^1$-smoothness we may assume that $T_\delta + (\varepsilon,0,\dots,0)\subset M^-$ for any small
$\varepsilon >0.$ Since $M^-$ is taut near $0$ (because of $\mathcal C^1$-smoothness and pseudoconvexity), then
$T_\delta\subset M.$

Hence $M_\delta=M\cap\D^n(0,\delta)$ is foliated by affine complex hyperplanes.
The set $M_\delta\cap(\C \times \{(0,\dots,0)\})$ is parametrized
by $-g(t)+it$, where $g$ is a convex, nonnegative $\mathcal C^1$-smooth function
with $g(0)=g'(0)=0$.
Then there exists a continuous map $\alpha:(-\delta,\delta)\mapsto\C^{n-1}$
so that $M_\delta$ is parametrized by
\begin{equation}
\label{param}
\left( -g(t)+it+ \alpha(t) \cdot z' , z' \right),\quad z'=(z_2,\dots,z_n).
\end{equation}
Since $M_\delta \subset \{\Re z_1\le 0\}$, we have (shrinking $\delta$ again)
$$
-g(t) + \Re ( \alpha(t) \cdot z' ) \le 0\mbox{ for any }z'\in\D^{n-1}(0,\delta).
$$
This implies that $\| \alpha (t) \| \le g(t)/\delta$.

Now consider more generally a point $p_0:=\left( -g(t_0)+it_0 , 0 \right)\in M$.
In the new coordinates
$$
\tilde z_1 = (1-ig'(t_0))\left[ (z_1 + g(t_0) - i t_0)- \alpha(t_0) \cdot z' \right],\quad
\tilde  z' =  z',
$$
one can check that the real tangent hyperplane to $M$ at $p_0$ is $\{\Re \tilde z_1=0\}$.
The equation of the hyperplane given by \eqref{param} becomes
$$
\tilde z_1 = (1-ig'(t_0)) \left[ (-g(t) + g(t_0) + i (t-t_0)+ (\alpha(t)-\alpha(t_0)) \cdot
\tilde z' \right],
$$
and using the fact that $M \subset \{\Re \tilde z_1 \le 0\}$, for $\|\tilde z'\|\le \delta_1$,
\begin{multline*}
\sqrt{1+g'(t_0)^2} \| \alpha(t)-\alpha(t_0) \| \le
\left[ g(t) - g(t_0) - (t-t_0)g'(t_0)\right]/\delta
\\
= (t-t_0) (g'(t_1)-g'(t_0))/\delta
\end{multline*}
for some value $t_1$ between $t_0$ and $t$, by the Mean Value Theorem.  Using the
uniform continuity of $g'$, we find some interval around $0$, where for any $\varepsilon >0$,
there exists $\eta >0$ such that $|t-t_0|\le \eta$ implies
$\| \alpha(t)-\alpha(t_0) \|/|t-t_0| \le \varepsilon$.  We deduce that $\alpha$
must be constant over that interval which yields the product structure that was claimed.
\end{proof}

\noindent{\bf Remarks.} 1. The argument at the beginning of the proof shows that
if the boundary of a linearly convex domain is locally Levi-flat,
then it too is foliated by complex hyperplanes.
On the other hand, the conclusion of this proposition does not hold in the
linearly convex case, even up to transformations preserving affine complex
hyperplanes: the boundary of symmetrized bidisc
$\Bbb G_2=\{z\in\C^2:|z_1-{\overline z_1}z_2|+|z_2|^2<1\}$ locally contains the tangent line
at any smooth point but $\Bbb G_2$ is not locally fractional linearly equivalent to a Cartesian product.

2. The result does not extend to the case where the Levi form has positive constant rank (strictly less than the maximal rank $n-1$).  
For example, consider the tube hypersurface over the cone $x_1x_3=x_2^2$ ($x_1>0$), i.e.
the hypersurface in $\C^3$ given by
$$
M:=\left\{ z \in \C^3 : \rho(z):= (\Re z_2)^2-(\Re z_1)(\Re z_3)=0, \Re z_1 >0 \right\}.
$$
Then $\{\rho <0<\Re z_1\}$ is convex, the Levi form of $\rho$ is semi definite positive and
of constant rank $1$ on $M$, and $M$ is foliated by portions of the affine complex lines
$$
\{ (a^2 \zeta + i b_1, a \zeta + i b_2, \zeta), \zeta \in \C\},\quad a, b_1, b_2 \in \mathbb R,
$$
which of course are not all parallel to each other.

\end{document}